\documentclass[10pt]{article}
\usepackage{amsmath}
\usepackage{amssymb}
\usepackage{amsthm}
\usepackage{mathrsfs}
\usepackage[compress,noadjust]{cite}
\numberwithin{equation}{section}

\begin{document}
\title{Parametric Marcinkiewicz integrals on weighted Hardy and weak Hardy spaces}
\author{Hua Wang \footnote{E-mail address: wanghua@pku.edu.cn.}\\
\footnotesize{Department of Mathematics, Zhejiang University, Hangzhou 310027, P. R. China}}
\date{}
\maketitle

\begin{abstract}
Let $0<\rho<n$ and $\mu^{\rho}_{\Omega}$ be the parametric Marcinkiewicz integral. In this paper, by using the atomic decomposition theory of weighted Hardy and weak Hardy spaces, we will obtain the boundedness properties of $\mu^{\rho}_\Omega$ on these spaces, under the Lipschitz condition imposed on the kernel $\Omega$.\\
MSC(2010): 42B25; 42B30\\
Keywords: Parametric Marcinkiewicz integrals; weighted Hardy spaces; weighted weak Hardy spaces; Lipschitz condition; $A_p$ weights; atomic decomposition
\end{abstract}

\section{Introduction}

Suppose that $S^{n-1}$ is the unit sphere in $\mathbb R^n$($n\ge2$) equipped with the normalized Lebesgue measure $d\sigma$. Let $\Omega$ be a homogeneous function of degree zero on $\mathbb R^n$ satisfying $\Omega\in L^1(S^{n-1})$ and
\begin{equation}
\int_{S^{n-1}}\Omega(x')\,d\sigma(x')=0,
\end{equation}
where $x'=x/{|x|}$ for any $x\neq0$. For $0<\rho<n$, in 1960, H\"ormander \cite{hor} defined the parametric Marcinkiewicz integral operator $\mu^{\rho}_{\Omega}$ of higher dimension as follows.
\begin{equation}
\mu^{\rho}_{\Omega}(f)(x)=\left(\int_0^\infty\big|F^{\rho}_{\Omega,t}(x)\big|^2\frac{dt}{t^{2\rho+1}}\right)^{1/2},
\end{equation}
where
\begin{equation}
F^{\rho}_{\Omega,t}(x)=\int_{|x-y|\le t}\frac{\Omega(x-y)}{|x-y|^{n-\rho}}f(y)\,dy.
\end{equation}
When $\rho=1$, we shall denote $\mu^1_\Omega$ simply by $\mu_\Omega$. This operator $\mu_\Omega$ was first introduced by Stein in \cite{stein1}. He proved that if $\Omega\in Lip_\alpha(S^{n-1})$ ($0<\alpha\le1$), then $\mu_\Omega$ is the operator of strong type $(p,p)$ for $1<p\le2$ and of weak type $(1,1)$. Here, we say that $\Omega\in Lip_\alpha(S^{n-1})$ if
\begin{equation}
\big|\Omega(x')-\Omega(y')\big|\le C|x'-y'|^\alpha, \quad x',y'\in S^{n-1}.
\end{equation}
In \cite{benedek}, Benedek, Calder\'on and Panzone showed that if $\Omega$ is continuously differentiable on $S^{n-1}$, then $\mu_\Omega$ is of strong type $(p,p)$ for all $1<p<\infty$. In 1990, Torchinsky and Wang \cite{torchinsky} considered the weighted case and proved that if $\Omega\in Lip_\alpha(S^{n-1})$, $0<\alpha\le1$, then for all $1<p<\infty$ and $w\in A_p$ (Muckenhoupt weight class), $\mu_\Omega$ is bounded on $L^p_w(\mathbb R^n)$. On the other hand, in 1960, H\"ormander \cite{hor} showed that if $\Omega\in Lip_\alpha(S^{n-1})$ ($0<\alpha\le1$), then for $0<\rho<n$, $\mu^{\rho}_\Omega$ is of strong type $(p,p)$ for all $1<p<\infty$. It is well known that the Littlewood-Paley $g$-function is a very important tool in harmonic analysis and the parametric Marcinkiewicz integral is essentially a Littlewood-Paley $g$-function. Therefore, many authors have been interested in studying the boundedness properties of $\mu^{\rho}_\Omega$ on various function spaces, one can see \cite{aliev,ding5,fang,yabuta} and the references therein for further details.

In \cite{sato1}, Sato established the following weighted $L^p$ boundedness of $\mu^{\rho}_\Omega$ for all $0<\rho<n$ (see also \cite{shi}).

\newtheorem*{thma}{Theorem A}

\begin{thma}
Let $0<\rho<n$ and $\Omega\in L^\infty(S^{n-1})$. If $w\in A_p$, $1<p<\infty$, then there exists a constant $C>0$ independent of $f$ such that
\begin{equation*}
\big\|\mu^{\rho}_{\Omega}(f)\big\|_{L^p_w}\le C\|f\|_{L^p_w}.
\end{equation*}
\end{thma}

The main purpose of this paper is to discuss the boundedness properties of parametric Marcinkiewicz integrals $\mu^{\rho}_\Omega$ ($0<\rho<n$) on the weighted Hardy and weak Hardy spaces (see Section 2 for the definitions). We now present our main results as follows.

\newtheorem{theorem}{Theorem}[section]

\begin{theorem}
Let $0<\rho<n$, $0<\alpha\le1$ and $\Omega\in Lip_\alpha(S^{n-1})$. If $n/{(n+\alpha)}<p\le1$ and $w\in A_{p(1+\frac{\alpha}{n})}$, then there exists a
constant $C>0$ independent of $f$ such that
\begin{equation*}
\big\|\mu^{\rho}_{\Omega}(f)\big\|_{L^p_w}\le C\|f\|_{H^p_w}.
\end{equation*}
\end{theorem}

\begin{theorem}
Let $0<\rho<n$, $0<\alpha<1$ and $\Omega\in Lip_\alpha(S^{n-1})$. If $p=n/{(n+\alpha)}$ and $w\in A_1$, then there exists a
constant $C>0$ independent of $f$ such that
\begin{equation*}
\big\|\mu^{\rho}_{\Omega}(f)\big\|_{WL^p_w}\le C\|f\|_{H^p_w}.
\end{equation*}
\end{theorem}

\begin{theorem}
Let $0<\rho<n$, $0<\alpha\le1$ and $\Omega\in Lip_\alpha(S^{n-1})$. If $n/{(n+\alpha)}<p\le1$ and $w\in A_{p(1+\frac{\alpha}{n})}$, then there exists a
constant $C>0$ independent of $f$ such that
\begin{equation*}
\big\|\mu^{\rho}_{\Omega}(f)\big\|_{WL^p_w}\le C\|f\|_{WH^p_w}.
\end{equation*}
\end{theorem}

\section{Notations and preliminaries}

Let us first recall some standard definitions and notations of $A_p$ weights. The classical $A_p$ weight theory was first introduced by Muckenhoupt in the study of weighted $L^p$ boundedness of Hardy-Littlewood maximal functions in \cite{muckenhoupt}. Let $w$ be a nonnegative, locally integrable function defined on $\mathbb R^n$; all cubes are assumed to have their sides parallel to the coordinate axes. For $1<p<\infty$, a weight function $w$ is said to belong to $A_p$, if there is a constant $C>0$ such that for every cube $Q\subseteq \mathbb R^n$,
\begin{equation}
\left(\frac1{|Q|}\int_Q w(x)\,dx\right)\left(\frac1{|Q|}\int_Q w(x)^{-1/{(p-1)}}\,dx\right)^{p-1}\le C,
\end{equation}
where $|Q|$ denotes the Lebesgue measure of $Q$. For the case $p=1$, $w\in A_1$, if there is a constant $C>0$ such that for every cube $Q\subseteq \mathbb R^n$,
\begin{equation}
\frac1{|Q|}\int_Q w(x)\,dx\le C\cdot\underset{x\in Q}{\mbox{ess\,inf}}\;w(x).
\end{equation}
A weight function $w\in A_\infty$ if it satisfies the $A_p$ condition for some $1<p<\infty$. It is well known that if $w\in A_p$ with $1<p<\infty$, then $w\in A_r$ for all $r>p$, and $w\in A_q$ for some $1<q<p$. We thus write $q_w\equiv\inf\{q>1:w\in A_q\}$ to denote the critical index of $w$. Given a cube $Q$ and $\lambda>0$, $\lambda Q$ stands for the cube with the same center as $Q$ whose side length is $\lambda$ times that of $Q$. $Q=Q(x_0,r)$ denotes the cube centered at $x_0$ with side length $r$. For a weight function $w$ and a measurable set $E$, we set the weighted measure of $E$ by $w(E)$, where $w(E)=\int_E w(x)\,dx$.

We state the following results that will be used later on.

\newtheorem{lemma}[theorem]{Lemma}

\begin{lemma}[\cite{garcia2}]
Let $w\in A_q$ with $q\ge1$. Then, for any cube $Q$, there exists an absolute constant $C>0$ such that
$$w(2Q)\le C\,w(Q).$$
In general, for any $\lambda>1$, we have
$$w(\lambda Q)\le C\cdot\lambda^{nq}w(Q),$$
where $C$ does not depend on $Q$ or $\lambda$.
\end{lemma}

\begin{lemma}[\cite{garcia2}]
Let $w\in A_q$ with $q>1$. Then, for all $r>0$, there exists a constant $C>0$ independent of $r$ such that
\begin{equation*}
\int_{|x|\ge r}\frac{w(x)}{|x|^{nq}}\,dx\le C\cdot r^{-nq}w\big(Q(0,2r)\big).
\end{equation*}
\end{lemma}

\begin{lemma}[\cite{garcia2}]
Let $w\in A_q$ with $q\ge1$. Then there exists an absolute constant $C>0$ such that
\begin{equation*}
C\cdot\left(\frac{|E|}{|Q|}\right)^q\le\frac{w(E)}{w(Q)}
\end{equation*}
for any measurable subset $E$ of a cube $Q$.
\end{lemma}

Given a weight function $w$ on $\mathbb R^n$, for $0<p<\infty$, we denote by $L^p_w(\mathbb R^n)$ the weighted space of all functions $f$ satisfying
\begin{equation}
\|f\|_{L^p_w}=\bigg(\int_{\mathbb R^n}|f(x)|^pw(x)\,dx\bigg)^{1/p}<\infty.
\end{equation}
When $p=\infty$, $L^\infty_w(\mathbb R^n)$ will be taken to mean $L^\infty(\mathbb R^n)$, and we set
\begin{equation}
\|f\|_{L^\infty_w}=\|f\|_{L^\infty}=\underset{x\in\mathbb R^n}{\mbox{ess\,sup}}\,|f(x)|<\infty.
\end{equation}
We also let $WL^p_w(\mathbb R^n)$ denote the weighted weak $L^p$ space of all those measurable functions $f$ which satisfy
\begin{equation}
\|f\|_{WL^p_w}=\sup_{\lambda>0}\lambda\cdot w\big(\big\{x\in\mathbb R^n:|f(x)|>\lambda \big\}\big)^{1/p}<\infty.
\end{equation}

We write $\mathscr S(\mathbb R^n)$ to denote the Schwartz space of all rapidly decreasing infinitely differentiable functions and $\mathscr S'(\mathbb R^n)$ to denote the space of all tempered distributions, i.e., the topological dual of $\mathscr S(\mathbb R^n)$. As we know, for any $0<p\le1$, the weighted Hardy spaces $H^p_w(\mathbb R^n)$ can be defined in terms of maximal functions.
Let $\varphi$ be a function in $\mathscr S(\mathbb R^n)$ satisfying $\int_{\mathbb R^n}\varphi(x)\,dx=1$.
Set
\begin{equation*}
\varphi_t(x)=t^{-n}\varphi(x/t),\quad t>0,\;x\in\mathbb R^n.
\end{equation*}
We will define the radial maximal function $M_\varphi f(x)$ by
\begin{equation*}
M_\varphi f(x)=\sup_{t>0}\big|(\varphi_t*f)(x)\big|.
\end{equation*}
Then the weighted Hardy space $H^p_w(\mathbb R^n)$ consists of those tempered distributions $f\in\mathscr S'(\mathbb R^n)$ for which
$M_\varphi f\in L^p_w(\mathbb R^n)$ with $\|f\|_{H^p_w}=\|M_\varphi f\|_{L^p_w}$. The real-variable theory of weighted Hardy spaces has been extensively investigated by many authors. For example, Garcia-Cuerva \cite{garcia1} studied the atomic decomposition and the dual spaces of $H^p_w$ for $0<p\le1$. The molecular characterization of $H^p_w$ for $0<p\le1$ was given by Lee and Lin \cite{lee2}. For more information about the continuity properties of some operators on weighted Hardy spaces, the reader is referred to \cite{ding7,ding8,lee1,lee3,li,lin}.

In this article, we will use Garcia-Cuerva's atomic decomposition theory for weighted Hardy spaces in \cite{garcia1,stomberg}. We characterize weighted Hardy spaces in terms of atoms in the following way.

Let $0<p\le1\le q\le\infty$ and $p\ne q$ such that $w\in A_q$ with critical index $q_w$. Set [\,$\cdot$\,] the greatest integer function. For $s\in \mathbb Z_+$ satisfying $s\ge N=[n({q_w}/p-1)],$ a real-valued function $a(x)$ is called a ($p,q,s$)-atom centered at $x_0$ with respect to $w$ (or a $w$-($p,q,s$)-atom centered at $x_0$) if the following conditions are satisfied:

(a) $a\in L^q_w(\mathbb R^n)$ and is supported in a cube $Q$ centered at $x_0$;

(b) $\|a\|_{L^q_w}\le w(Q)^{1/q-1/p}$;

(c) $\int_{\mathbb R^n}a(x)x^\alpha\,dx=0$ for every multi-index $\alpha$ with $|\alpha|\le s$.

\begin{theorem}
Let $0<p\le1\le q\le\infty$ and $p\ne q$ such that $w\in A_q$ with critical index $q_w$. For each $f\in H^p_w(\mathbb R^n)$, there exist a sequence \{$a_j$\} of $w$-$(p,q,s)$-atoms and a sequence \{$\lambda_j$\} of real numbers with $\sum_j|\lambda_j|^p\le C\|f\|^p_{H^p_w}$ such that $f=\sum_j\lambda_j a_j$ both in the sense of distributions and in the $H^p_w$ norm.
\end{theorem}

Let us now turn to the weighted weak Hardy spaces, which are good substitutes for the weighted
Hardy spaces in the study of the boundedness of some operators. The (unweighted) weak $H^p$ spaces have first appeared in the work of Fefferman, Rivi\`ere and Sagher \cite{cfefferman}, which are the intermediate spaces between two Hardy spaces through the real method of interpolation. The atomic decomposition characterization of weak $H^1$ space on $\mathbb R^n$ was given by Fefferman and Soria in \cite{rfefferman}. Later, Liu \cite{liu} established the weak $H^p$ spaces on homogeneous groups for the whole range $0<p\le1$. The corresponding results related to $\mathbb R^n$ can be found in \cite{lu}. For the boundedness properties of some operators on weak Hardy spaces, we refer the readers to \cite{ding1,ding2,ding3,ding4,ding5,ding6,liu2,tao}. In 2000, Quek and Yang \cite{quek} introduced the weighted weak Hardy spaces $WH^p_w(\mathbb R^n)$ and established their atomic decompositions. Moreover, by using the atomic decomposition theory of $WH^p_w(\mathbb R^n)$, Quek and Yang \cite{quek} also obtained the boundedness of Calder\'on-Zygmund type operators on these weighted spaces.

Let $w\in A_\infty$, $0<p\le1$ and $N=[n(q_w/p-1)]$. Define
\begin{equation*}
\mathscr A_{N,w}=\Big\{\varphi\in\mathscr S(\mathbb R^n):\sup_{x\in\mathbb R^n}\sup_{|\alpha|\le N+1}(1+|x|)^{N+n+1}\big|D^\alpha\varphi(x)\big|\le1\Big\},
\end{equation*}
where $\alpha=(\alpha_1,\dots,\alpha_n)\in(\mathbb N\cup\{0\})^n$, $|\alpha|=\alpha_1+\dots+\alpha_n$, and
\begin{equation*}
D^\alpha\varphi=\frac{\partial^{|\alpha|}\varphi}{\partial x^{\alpha_1}_1\cdots\partial x^{\alpha_n}_n}.
\end{equation*}
For any given $f\in\mathscr S'(\mathbb R^n)$, the grand maximal function of $f$ is defined by
\begin{equation*}
G_w f(x)=\sup_{\varphi\in\mathscr A_{N,w}}\sup_{|y-x|<t}\big|(\varphi_t*f)(y)\big|.
\end{equation*}
Then we can define the weighted weak Hardy space $WH^p_w(\mathbb R^n)$ by $WH^p_w(\mathbb R^n)=\big\{f\in\mathscr S'(\mathbb R^n):G_w f\in WL^p_w(\mathbb R^n)\big\}$. Moreover, we set $\|f\|_{WH^p_w}=\|G_w f\|_{WL^p_w}$.

\begin{theorem}[\cite{quek}]
Let $0<p\le1$ and $w\in A_\infty$. For every $f\in WH^p_w(\mathbb R^n)$, there exists a sequence of bounded measurable functions $\{f_k\}_{k=-\infty}^\infty$ such that

$(i)$ $f=\sum_{k=-\infty}^\infty f_k$ in the sense of distributions.

$(ii)$ Each $f_k$ can be further decomposed into $f_k=\sum_i b^k_i$, where $\{b^k_i\}$ satisfies

\quad $(a)$ Each $b^k_i$ is supported in a cube $Q^k_i$ with $\sum_{i}w\big(Q^k_i\big)\le c2^{-kp}$, and $\sum_i\chi_{Q^k_i}(x)\le c$. Here $\chi_E$ denotes the characteristic function of the set $E$ and $c\sim\big\|f\big\|_{WH^p_w}^p;$

\quad $(b)$ $\big\|b^k_i\big\|_{L^\infty}\le C2^k$, where $C>0$ is independent of $i$ and $k\,;$

\quad $(c)$ $\int_{\mathbb R^n}b^k_i(x)x^\alpha\,dx=0$ for every multi-index $\alpha$ with $|\alpha|\le[n({q_w}/p-1)]$.

Conversely, if $f\in\mathscr S'(\mathbb R^n)$ has a decomposition satisfying $(i)$ and $(ii)$, then $f\in WH^p_w(\mathbb R^n)$. Moreover, we have $\big\|f\big\|_{WH^p_w}^p\sim c.$
\end{theorem}

Throughout this article $C$ always denotes a positive constant, which is independent of the main parameters and not necessarily the same at each occurrence. Moreover, we use $A\sim B$ to mean the equivalence of $A$ and $B$; that is, there exist two positive constants $C_1$ and $C_2$ independent of $A$, $B$ such that $C_1A\le B\le C_2A$.

\section{Proof of Theorem 1.1}

\begin{proof}[Proof of Theorem 1.1]
Set $q=p(1+\frac{\alpha}{n})$. Then by our assumption, we have $[n({q_w}/p-1)]=0$ provided that $w\in A_q$. In view of Theorem 2.4 and Theorem A, it suffices to show that for any $w$-$(p,q,0)$-atom $a$, there exists a constant $C>0$ independent of $a$ such that $\big\|\mu^{\rho}_\Omega(a)\big\|_{L^p_w}\le C$. Let $a(x)$ be a $w$-$(p,q,0)$-atom with supp\,$a\subseteq Q=Q(x_0,r)$, and let $Q^*=2\sqrt nQ$. We write
\begin{equation*}
\begin{split}
\big\|\mu^{\rho}_\Omega(a)\big\|^p_{L^p_w}&=\int_{Q^*}\big|\mu^{\rho}_\Omega(a)(x)\big|^pw(x)\,dx
+\int_{(Q^*)^c}\big|\mu^{\rho}_\Omega(a)(x)\big|^pw(x)\,dx\\
&=I_1+I_2.
\end{split}
\end{equation*}
For the term $I_1$, by using H\"older's inequality with exponent $s=q/p>1$, the size condition of atom $a$, Lemma 2.1 and Theorem A, we have
\begin{equation*}
\begin{split}
I_1&\le\bigg(\int_{Q^*}\big|\mu^{\rho}_\Omega(a)(x)\big|^qw(x)\,dx\bigg)^{p/q}
\bigg(\int_{Q^*}w(x)\,dx\bigg)^{1-p/q}\\
&\le \big\|\mu^{\rho}_\Omega(a)\big\|^p_{L^q_w}\big[w(Q^*)\big]^{1-p/q}\\
&\le C\cdot\big\|\mu^{\rho}_\Omega(a)\big\|^p_{L^q_w}\big[w(Q)\big]^{1-p/q}\\
&\le C\cdot\|a\|^p_{L^q_w}\big[w(Q)\big]^{1-p/q}\\
&\le C\notag.
\end{split}
\end{equation*}
Let us now turn to estimate the other term $I_2$. For $0<\rho<n$, if we set $\psi^{\rho}(x)=\Omega(x)|x|^{-n+\rho}\chi_{\{|x|\le1\}}(x)$, then for given atom $a$, we can see that
\begin{equation}
\mu^{\rho}_\Omega(a)(x)=\left(\int_0^\infty\big|(\psi^{\rho}_t*a)(x)\big|^2\frac{dt}{t}\right)^{1/2},
\end{equation}
where $\psi^{\rho}_t(x)=t^{-n}\psi^{\rho}(x/t)$. Thus, by the vanishing moment condition of atom $a$, we deduce that
\begin{equation*}
\begin{split}
\big|(\psi^{\rho}_t*a)(x)\big|=\,&\frac{1}{t^{\rho}}\cdot
\left|\int_Q\bigg[\frac{\Omega(x-y)}{|x-y|^{n-\rho}}-\frac{\Omega(x-x_0)}{|x-x_0|^{n-\rho}}\bigg]a(y)\,dy\right|\\
\le\,&\frac{1}{t^{\rho}}\cdot\int_Q\bigg|\frac{1}{|x-y|^{n-\rho}}-\frac{1}{|x-x_0|^{n-\rho}}\bigg|
\big|\Omega(x-y)a(y)\big|\,dy\\
&+\frac{1}{t^{\rho}}\cdot\int_Q\frac{|\Omega(x-y)-\Omega(x-x_0)|}{|x-x_0|^{n-\rho}}|a(y)|\,dy\\
=\,&\mbox{\upshape I+II}.
\end{split}
\end{equation*}
Clearly, the condition $\Omega\in Lip_\alpha(S^{n-1})$($0<\alpha\le1$) implies that $\Omega\in L^\infty(S^{n-1})$. Notice also that when $x\in(Q^*)^c$ and $y\in Q$, then we get $|x-y|\sim |x-x_0|$. Hence, we apply the mean value theorem to obtain
\begin{align}
\mbox{\upshape I}&\le \frac{1}{t^{\rho}}\cdot\|\Omega\|_{L^\infty}\int_Q
\bigg|\frac{1}{|x-y|^{n-\rho}}-\frac{1}{|x-x_0|^{n-\rho}}\bigg||a(y)|\,dy\notag\\
&\le C\cdot\frac{r}{t^{\rho}|x-x_0|^{n-\rho+1}}\int_Q|a(y)|\,dy.
\end{align}
On the other hand, from the definition of $\Omega\in Lip_\alpha(S^{n-1})$($0<\alpha\le1$), we can easily check that
\begin{align}
\big|\Omega(x-y)-\Omega(x-x_0)\big|&\le\left|\Omega\Big(\frac{x-y}{|x-y|}\Big)
-\Omega\Big(\frac{x-x_0}{|x-x_0|}\Big)\right|\notag\\
&\le C\left|\frac{x-y}{|x-y|}-\frac{x-x_0}{|x-x_0|}\right|^\alpha\notag\\
&\le C\left(\frac{|y-x_0|}{|x-x_0|}\right)^\alpha.
\end{align}
Substituting the above inequality (3.3) into the term II, then we get
\begin{equation}
\mbox{\upshape II}\le C\cdot\frac{r^\alpha}{t^{\rho}|x-x_0|^{n-\rho+\alpha}}\int_Q|a(y)|\,dy.
\end{equation}
Denote the conjugate exponent of $q>1$ by $q'=q/{(q-1)}$. Then it follows from H\"older's inequality, the $A_q$ condition and the size condition of atom $a$ that
\begin{align}
\int_Q\big|a(y)\big|\,dy&\le\left(\int_Q\big|a(y)\big|^qw(y)\,dy\right)^{1/q}\left(\int_Q w(y)^{-{q'}/q}\,dy\right)^{1/q'}\notag\\
&\le C\cdot\|a\|_{L^q_w}\left(\frac{|Q|^q}{w(Q)}\right)^{1/q}\notag\\
&\le C\cdot\frac{|Q|}{[w(Q)]^{1/p}}.
\end{align}
Observe that supp\,$\psi^{\rho}\subseteq\{x\in\mathbb R^n:|x|\le1\}$, then for any $y\in Q$ and $x\in(Q^*)^c$, we have
\begin{equation}
t\ge|x-y|\ge|x-x_0|-|y-x_0|\ge\frac{|x-x_0|}{2}.
\end{equation}
From the above estimates (3.5) and (3.6), it follows that for any $x\in(Q^*)^c$,
\begin{align}
\big|\mu^{\rho}_\Omega(a)(x)\big|&\le C\Bigg(\frac{r^{n+1}}{|x-x_0|^{n-\rho+1}[w(Q)]^{1/p}}+
\frac{r^{n+\alpha}}{|x-x_0|^{n-\rho+\alpha}[w(Q)]^{1/p}}\Bigg)
\bigg(\int_{\frac{|x-x_0|}{2}}^\infty\frac{dt}{t^{2\rho+1}}\bigg)^{1/2}\notag\\
&\le C\Bigg(\frac{r^{n+1}}{|x-x_0|^{n+1}[w(Q)]^{1/p}}+
\frac{r^{n+\alpha}}{|x-x_0|^{n+\alpha}[w(Q)]^{1/p}}\Bigg).
\end{align}
Hence
\begin{equation*}
\begin{split}
I_2&\le C\cdot\frac{r^{(n+\alpha)p}}{w(Q)}\int_{(Q^*)^c}\frac{w(x)}{|x-x_0|^{(n+\alpha)p}}\,dx
+C\cdot\frac{r^{(n+1)p}}{w(Q)}\int_{(Q^*)^c}\frac{w(x)}{|x-x_0|^{(n+1)p}}\,dx\\
&=\mbox{\upshape III+IV}.
\end{split}
\end{equation*}
Recall that $q=p(1+\frac{\alpha}{n})$ and $w\in A_q$. By using Lemma 2.1 and Lemma 2.2, we thus obtain
\begin{equation*}
\begin{split}
\mbox{\upshape III}&=C\cdot\frac{r^{(n+\alpha)p}}{w(Q)}\int_{|y|\ge\sqrt n r}\frac{w_1(y)}{|y|^{nq}}\,dy\\
&\le C\cdot\frac{r^{(n+\alpha)p}}{w(Q)}\cdot r^{-nq}w_1\big(Q(0,r)\big)\\
&=C\cdot\frac{r^{(n+\alpha)p}}{w(Q)}\cdot r^{-nq}w(Q)\\
&\le C,
\end{split}
\end{equation*}
where $w_1(x)=w(x+x_0)$ is the translation of $w(x)$. It is obvious that $w_1\in A_q$ for $w\in A_q$, $q>1$, and $q_{w_1}=q_w$. Since $w\in A_{p(1+\frac{\alpha}{n})}$ and $0<\alpha\le1$, then we have $w\in A_{p(1+\frac{1}{n})}$. Therefore, by using the same arguments as above, we can also prove that
\begin{equation*}
\mbox{\upshape IV}\le C.
\end{equation*}
Summing up the above estimates for $I_1$ and $I_2$, we then complete the proof of Theorem 1.1.
\end{proof}

\section{Proof of Theorem 1.2}

In order to prove our main result of this section, we shall need the following superposition principle on the weighted weak type estimates.

\begin{lemma}
Let $w\in A_1$ and $0<p<1$.
If a sequence of measurable functions $\{f_j\}$ satisfy
\begin{equation*}
w\big(\big\{x\in\mathbb R^n:|f_j(x)|>\alpha\big\}\big)\le \alpha^{-p} \quad\mbox{for all}\;\, j\in\mathbb Z
\end{equation*}
and
\begin{equation*}
\sum_{j\in\mathbb Z}|\lambda_j|^p\le1,
\end{equation*}
then we obtain that $\sum_j\lambda_jf_j(x)$ is absolutely convergent almost everywhere and
\begin{equation*}
w\Big(\Big\{x\in\mathbb R^n:\Big|\sum_j\lambda_jf_j(x)\Big|>\alpha\Big\}\Big)\le\frac{2-p}{1-p}\cdot\alpha^{-p}.
\end{equation*}
\end{lemma}
\begin{proof}
The proof of this lemma is similar to the corresponding result for the unweighted case which can be found in \cite{stein2}. See also [24, p. 123].
\end{proof}

We are now ready to give the proof of Theorem 1.2.

\begin{proof}[Proof of Theorem 1.2]
We first observe that for $w\in A_1$ and $p=n/{(n+\alpha)}$, then $[n({q_w}/p-1)]=[\alpha]=0$. According to Theorem 2.3 and Lemma 4.1, it is enough for us to show that for any $w$-$(p,q,0)$-atom $a(x)$, there exists a constant $C>0$ independent of $a$ such that $\big\|\mu^{\rho}_{\Omega}(a)\big\|_{WL^p_w}\le C$. Let $a(x)$ be a $w$-$(p,q,0)$-atom centered at $x_0$ with supp\,$a\subseteq Q=Q(x_0,r)$, and let $Q^*=2\sqrt nQ$. Then for any fixed $\lambda>0$, we write
\begin{equation*}
\begin{split}
&\lambda^p\cdot w\big(\big\{x\in\mathbb R^n:\big|\mu^{\rho}_\Omega(a)(x)\big|>\lambda\big\}\big)\\
\le\,&\lambda^p\cdot w\big(\big\{x\in Q^*:\big|\mu^{\rho}_\Omega(a)(x)\big|>\lambda\big\}\big)
+\lambda^p\cdot w\big(\big\{x\in(Q^*)^c:\big|\mu^{\rho}_\Omega(a)(x)\big|>\lambda\big\}\big)\\
=\,&J_1+J_2.
\end{split}
\end{equation*}
Since $w\in A_1$, then $w\in A_q$ for any $1<q<\infty$. Applying Chebyshev's inequality, H\"older's inequality, Lemma 2.1, Theorem A and the size condition of atom $a$, we have
\begin{equation*}
\begin{split}
J_1&\le\int_{Q^*}\big|\mu^{\rho}_\Omega(a)(x)\big|^pw(x)\,dx\\
&\le\bigg(\int_{Q^*}\big|\mu^{\rho}_\Omega(a)(x)\big|^qw(x)\,dx\bigg)^{p/q}
\bigg(\int_{Q^*}w(x)\,dx\bigg)^{1-p/q}\\
&\le C\cdot\big\|\mu^{\rho}_\Omega(a)\big\|^p_{L^q_w}\big[w(Q)\big]^{1-p/q}\\
&\le C\cdot\|a\|^p_{L^q_w}\big[w(Q)\big]^{1-p/q}\\
&\le C.
\end{split}
\end{equation*}
For any $x\in(Q^*)^c$, in the proof of Theorem 1.1, we have already obtained the following pointwise inequality (see (3.7))
\begin{equation*}
\big|\mu^{\rho}_\Omega(a)(x)\big|\le C\Bigg(\frac{r^{n+\alpha}}{|x-x_0|^{n+\alpha}[w(Q)]^{1/p}}+
\frac{r^{n+1}}{|x-x_0|^{n+1}[w(Q)]^{1/p}}\Bigg).
\end{equation*}
Setting
\begin{equation*}
G(x)=\frac{r^{n+\alpha}}{|x-x_0|^{n+\alpha}[w(Q)]^{1/p}}\quad\mbox{and}\quad H(x)=\frac{r^{n+1}}{|x-x_0|^{n+1}[w(Q)]^{1/p}}.
\end{equation*}
Thus, in order to complete the proof of Theorem 1.2, we only need to prove that the following two inequalities hold.
\begin{equation}
\lambda^p\cdot w\big(\big\{x\in(Q^*)^c:\big|G(x)\big|>\lambda\big\}\big)\le C
\end{equation}
and
\begin{equation}
\lambda^p\cdot w\big(\big\{x\in(Q^*)^c:\big|H(x)\big|>\lambda\big\}\big)\le C.
\end{equation}
Let us start with the inequality (4.1). For any given $\lambda>0$, we are going to consider two cases. If
\begin{equation*}
\lambda\ge\frac{r^{n+\alpha}}{{(\sqrt nr)}^{n+\alpha}[w(Q)]^{1/p}},
\end{equation*}
then for any $x\in(Q^*)^c$, we have $|x-x_0|\ge\sqrt nr$. Hence, we can easily verify that
\begin{equation*}
\big\{x\in(Q^*)^c:\big|G(x)\big|>\lambda\big\}=\O.
\end{equation*}
Therefore, in this case, the inequality
\begin{equation*}
\lambda^p\cdot w\big(\big\{x\in(Q^*)^c:\big|G(x)\big|>\lambda\big\}\big)\le C
\end{equation*}
holds trivially. Now suppose that
\begin{equation*}
\lambda<\frac{r^{n+\alpha}}{{(\sqrt nr)}^{n+\alpha}[w(Q)]^{1/p}}.
\end{equation*}
If we take $R=\frac{r}{\lambda^{p/n}[w(Q)]^{1/n}}$, then it is not difficult to check that $R\ge\sqrt nr\ge r$ and
\begin{equation}
\big\{x\in(Q^*)^c:\big|G(x)\big|>\lambda\big\}\subseteq\big\{x\in\mathbb R^n:|x-x_0|<R\big\}\subseteq Q(x_0,2R).
\end{equation}
Since $w\in A_1$, then by Lemma 2.3, we can get (below, $\widetilde C$ is an absolute constant)
\begin{equation*}
\widetilde C \cdot\frac{|Q(x_0,r)|}{|Q(x_0,R)|}\le\frac{w(Q(x_0,r))}{w(Q(x_0,R))},
\end{equation*}
which implies
\begin{equation*}
\begin{split}
w\big(Q(x_0,R)\big)&\le\frac{R^n\cdot w(Q)}{\widetilde C\cdot r^n}\\
&\le\frac{1}{\widetilde C\cdot\lambda^p}.
\end{split}
\end{equation*}
Hence, it follows directly from (4.3) and Lemma 2.1 that
\begin{equation*}
\begin{split}
\lambda^p\cdot w\big(\big\{x\in(Q^*)^c:\big|G(x)\big|>\lambda\big\}\big)
&\le\lambda^p\cdot w\big(Q(x_0,2R)\big)\\
&\le C\cdot\lambda^p\cdot w\big(Q(x_0,R)\big)\\
&\le C\cdot{\widetilde C}^{-1}.
\end{split}
\end{equation*}
It remains to prove the inequality (4.2). Similarly, for any given $\lambda>0$, we will consider the following two cases. If
\begin{equation*}
\lambda\ge\frac{r^{n+1}}{{(\sqrt nr)}^{n+1}[w(Q)]^{1/p}},
\end{equation*}
then as before, we can also show that
\begin{equation*}
\big\{x\in(Q^*)^c:\big|H(x)\big|>\lambda\big\}=\O,
\end{equation*}
and so the following estimate holds trivially.
\begin{equation*}
\lambda^p\cdot w\big(\big\{x\in(Q^*)^c:\big|H(x)\big|>\lambda\big\}\big)\le C.
\end{equation*}
Now if instead we assume that
\begin{equation*}
\lambda<\frac{r^{n+1}}{{(\sqrt nr)}^{n+1}[w(Q)]^{1/p}}.
\end{equation*}
In this case, if we take $R'=\frac{r}{\lambda^{1/{(n+1)}}[w(Q)]^{1/{p(n+1)}}}$, then it is not difficult to verify that $R'\ge\sqrt nr\ge r$ and
\begin{equation}
\big\{x\in(Q^*)^c:\big|H(x)\big|>\lambda\big\}\subseteq\big\{x\in\mathbb R^n:|x-x_0|<R'\big\}\subseteq Q(x_0,2R').
\end{equation}
Recall that $p=n/{(n+\alpha)}$ and $0<\alpha<1$, then $1<p(1+\frac{1}{n})$. Since $w\in A_1$, then $w\in A_{p(1+\frac{1}{n})}$. Furthermore, by using Lemma 2.3 again, we can get (below, $\widetilde{\widetilde{C}}$ is an absolute constant)
\begin{equation*}
\widetilde{\widetilde{C}}\cdot\left(\frac{|Q(x_0,r)|}{|Q(x_0,R')|}\right)^{p(1+\frac{1}{n})}
\le\frac{w(Q(x_0,r))}{w(Q(x_0,R'))},
\end{equation*}
which in turn gives
\begin{equation*}
\begin{split}
w\big(Q(x_0,R')\big)&\le\frac{(R')^{p(n+1)}\cdot w(Q)}{\widetilde{\widetilde{C}}\cdot r^{p(n+1)}}\\
&\le\frac{1}{\widetilde{\widetilde{C}}\cdot\lambda^p}.
\end{split}
\end{equation*}
Therefore, by (4.4) and Lemma 2.1, we obtain
\begin{equation*}
\begin{split}
\lambda^p\cdot w\big(\big\{x\in(Q^*)^c:\big|H(x)\big|>\lambda\big\}\big)
&\le\lambda^p\cdot w\big(Q(x_0,2R')\big)\\
&\le C\cdot\lambda^p\cdot w\big(Q(x_0,R')\big)\\
&\le C\cdot{\widetilde{\widetilde C}}^{-1}.
\end{split}
\end{equation*}
Collecting all these estimates and then taking the supremum over all $\lambda>0$, we conclude the proof of Theorem 1.2.
\end{proof}

\section{Proof of Theorem 1.3}

\begin{proof}[Proof of Theorem 1.3]
The basic idea of the proof is taken from \cite{quek}. For any given $\lambda>0$, we may choose $k_0\in\mathbb Z$ such that $2^{k_0}\le\lambda<2^{k_0+1}$. For every $f\in WH^p_w(\mathbb R^n)$, then by Theorem 2.5, we can write
\begin{equation*}
f=\sum_{k=-\infty}^\infty f_k=\sum_{k=-\infty}^{k_0} f_k+\sum_{k=k_0+1}^\infty f_k:=F_1+F_2,
\end{equation*}
where $F_1=\sum_{k=-\infty}^{k_0} f_k=\sum_{k=-\infty}^{k_0}\sum_i b^k_i$, $F_2=\sum_{k=k_0+1}^\infty f_k=\sum_{k=k_0+1}^\infty\sum_i b^k_i$ and $\{b^k_i\}$ satisfies $(a)$--$(c)$ in Theorem 2.5. Then we have \begin{equation*}
\begin{split}
&\lambda^p\cdot w\big(\big\{x\in\mathbb R^n:\big|\mu^{\rho}_{\Omega}(f)(x)\big|>\lambda\big\}\big)\\
\le\,&\lambda^p\cdot w\big(\big\{x\in\mathbb R^n:\big|\mu^{\rho}_{\Omega}(F_1)(x)\big|>\lambda/2\big\}\big)+\lambda^p\cdot w\big(\big\{x\in\mathbb R^n:\big|\mu^{\rho}_{\Omega}(F_2)(x)\big|>\lambda/2\big\}\big)\\
=\,&K_1+K_2.
\end{split}
\end{equation*}
First we claim that the following inequality holds:
\begin{equation}
\big\|F_1\big\|_{L^2_w}\le C\cdot\lambda^{1-p/2}\big\|f\big\|^{p/2}_{WH^p_w}.
\end{equation}
In fact, since supp\,$b^k_i\subseteq Q^k_i=Q\big(x^k_i,r^k_i\big)$ and $\big\|b^k_i\big\|_{L^\infty}\le C 2^k$ by Theorem 2.5, then it follows from Minkowski's integral inequality that
\begin{equation*}
\begin{split}
\big\|F_1\big\|_{L^2_w}&\le\sum_{k=-\infty}^{k_0}\sum_i\big\|b^k_i\big\|_{L^2_w}\\
&\le\sum_{k=-\infty}^{k_0}\sum_i\big\|b^k_i\big\|_{L^\infty}w\big(Q^k_i\big)^{1/2}.
\end{split}
\end{equation*}
For each $k\in\mathbb Z$, by using the bounded overlapping property of the cubes $\{Q^k_i\}$ and the fact that $1-p/2>0$, we thus obtain
\begin{equation*}
\begin{split}
\big\|F_1\big\|_{L^2_w}&\le C\sum_{k=-\infty}^{k_0}2^k\Big(\sum_i w\big(Q^k_i\big)\Big)^{1/2}\\
&\le C\sum_{k=-\infty}^{k_0}2^{k(1-p/2)}\big\|f\big\|^{p/2}_{WH^p_w}\\
&\le C\sum_{k=-\infty}^{k_0}2^{(k-k_0)(1-p/2)}\cdot\lambda^{1-p/2}\big\|f\big\|^{p/2}_{WH^p_w}\\
&\le C\cdot\lambda^{1-p/2}\big\|f\big\|^{p/2}_{WH^p_w}.
\end{split}
\end{equation*}
By the hypothesis $w\in A_{p(1+\frac{\alpha}{n})}$ and $1<p(1+\frac{\alpha}{n})\le1+\frac{\alpha}{n}\le2$, then we have $w\in A_2$. Applying Chebyshev's inequality, Theorem A and the inequality (5.1), we can deduce that
\begin{equation*}
\begin{split}
K_1&\le \lambda^p\cdot\frac{4}{\lambda^2}\big\|\mu^{\rho}_{\Omega}(F_1)\big\|^2_{L^2_w}\\
&\le C\cdot\lambda^{p-2}\big\|F_1\big\|^2_{L^2_w}\\
&\le C\big\|f\big\|^{p}_{WH^p_w}.
\end{split}
\end{equation*}
Now we turn our attention to the estimate of $K_2$. We set
\begin{equation*}
A_{k_0}=\bigcup_{k=k_0+1}^\infty\bigcup_i \widetilde{Q^k_i},
\end{equation*}
where $\widetilde{Q^k_i}=Q\big(x^k_i,\tau^{{(k-k_0)}/{(n+\alpha)}}(2\sqrt n)r^k_i\big)$ and $\tau$ is a fixed positive number such that $1<\tau<2$. Thus, we can further decompose $K_2$ as
\begin{equation*}
\begin{split}
K_2&\le\lambda^p\cdot w\big(\big\{x\in A_{k_0}:\big|\mu^{\rho}_{\Omega}(F_2)(x)\big|>\lambda/2\big\}\big)+
\lambda^p\cdot w\big(\big\{x\in (A_{k_0})^c:\big|\mu^{\rho}_{\Omega}(F_2)(x)\big|>\lambda/2\big\}\big)\\
&=K'_2+K''_2.
\end{split}
\end{equation*}
Let us first deal with the term $K'_2$. Since $w\in A_{p(1+\frac{\alpha}{n})}$, then by Lemma 2.1, we can get
\begin{equation*}
\begin{split}
K'_2&\le\lambda^p\sum_{k=k_0+1}^\infty\sum_i w\Big(\widetilde{Q^k_i}\Big)\\
&\le C\cdot\lambda^p\sum_{k=k_0+1}^\infty\tau^{(k-k_0)p}\sum_i w\big(Q^k_i\big)\\
&\le C\big\|f\big\|^{p}_{WH^p_w}\sum_{k=k_0+1}^\infty\Big(\frac{\tau}{2}\Big)^{(k-k_0)p}\\
&\le C\big\|f\big\|^{p}_{WH^p_w}.
\end{split}
\end{equation*}
On the other hand, an application of Chebyshev's inequality leads to that
\begin{equation*}
\begin{split}
K''_2&\le 2^p\int_{(A_{k_0})^c}\big|\mu^{\rho}_{\Omega}(F_2)(x)\big|^pw(x)\,dx\\
&\le 2^p\sum_{k=k_0+1}^\infty\sum_i
\int_{\big(\widetilde{Q^k_i}\big)^c}\big|\mu^{\rho}_{\Omega}\big(b^k_i\big)(x)\big|^pw(x)\,dx.
\end{split}
\end{equation*}
As before, if we set $\psi^{\rho}(x)=\Omega(x)|x|^{-n+\rho}\chi_{\{|x|\le1\}}(x)$, then
\begin{equation}
\mu^{\rho}_\Omega\big(b^k_i\big)(x)=
\left(\int_0^\infty\big|\big(\psi^{\rho}_t*b^k_i\big)(x)\big|^2\frac{dt}{t}\right)^{1/2}.
\end{equation}
Let $q=p(1+\frac{\alpha}{n})$ for simplicity. Then for any $n/{(n+\alpha)}<p\le1$ and $w\in A_q$ with $q>1$, we can easily see that $[n(q_w/p-1)]=0$. Hence, for any $x\in\big(\widetilde{Q^k_i}\big)^c$, by the vanishing moment condition of $b^k_i$, we have
\begin{equation*}
\begin{split}
\big|\big(\psi^{\rho}_t*b^k_i\big)(x)\big|=\,&\frac{1}{t^{\rho}}\cdot
\left|\int_{Q^k_i}\bigg[\frac{\Omega(x-y)}{|x-y|^{n-\rho}}-\frac{\Omega(x-x^k_i)}{|x-x^k_i|^{n-\rho}}\bigg]
b^k_i(y)\,dy\right|\\
\le\,&\frac{1}{t^{\rho}}\cdot\int_{Q^k_i}
\bigg|\frac{1}{|x-y|^{n-\rho}}-\frac{1}{|x-x^k_i|^{n-\rho}}\bigg|\big|\Omega(x-y)b^k_i(y)\big|\,dy\\
&+\frac{1}{t^{\rho}}\cdot\int_{Q^k_i}\frac{|\Omega(x-y)-\Omega(x-x^k_i)|}{|x-x^k_i|^{n-\rho}}\big|b^k_i(y)\big|\,dy\\
=\,&\mbox{\upshape I+II}.
\end{split}
\end{equation*}
Since $\Omega\in Lip_\alpha(S^{n-1})$, $0<\alpha\le1$, then $\Omega\in L^\infty(S^{n-1})$. Note that for any $y\in Q^k_i$ and $x\in\big(\widetilde{Q^k_i}\big)^c$, then $|x-y|\sim |x-x^k_i|$. This fact together with the mean value theorem yields
\begin{align}
\mbox{\upshape I}&\le\frac{1}{t^{\rho}}\cdot\|\Omega\|_{L^\infty}\int_{Q^k_i}
\bigg|\frac{1}{|x-y|^{n-\rho}}-\frac{1}{|x-x^k_i|^{n-\rho}}\bigg|\big|b^k_i(y)\big|\,dy\notag\\
&\le C\cdot\big\|b^k_i\big\|_{L^\infty}\frac{(r^k_i)^{n+1}}{t^{\rho}|x-x^k_i|^{n-\rho+1}}.
\end{align}
For the term II, we still have $|x-y|\sim |x-x^k_i|$, when $y\in Q^k_i$ and $x\in\big(\widetilde{Q^k_i}\big)^c$. From the definition of $\Omega\in Lip_\alpha(S^{n-1})$, we can easily see that
\begin{align}
\big|\Omega(x-y)-\Omega(x-x^k_i)\big|&\le\left|\Omega\Big(\frac{x-y}{|x-y|}\Big)
-\Omega\Big(\frac{x-x^k_i}{|x-x^k_i|}\Big)\right|\notag\\
&\le C\left|\frac{x-y}{|x-y|}-\frac{x-x^k_i}{|x-x^k_i|}\right|^\alpha\notag\\
&\le C\left(\frac{|y-x^k_i|}{|x-x^k_i|}\right)^\alpha.
\end{align}
Substituting the above inequality (5.4) into the term II, then we get
\begin{equation}
\mbox{\upshape II}\le C\cdot\big\|b^k_i\big\|_{L^\infty}\frac{(r^k_i)^{n+\alpha}}{t^{\rho}|x-x^k_i|^{n-\rho+\alpha}}.
\end{equation}
When $y\in Q^k_i$ and $x\in\big(\widetilde{Q^k_i}\big)^c$, then a direct calculation shows that
\begin{equation}
t\ge|x-y|\ge\big|x-x^k_i\big|-\big|y-x^k_i\big|\ge\frac{|x-x^k_i|}{2}.
\end{equation}
Summarizing the above two estimates for I and II, for any $x\in \big(\widetilde{Q^k_i}\big)^c$, we have
\begin{equation*}
\begin{split}
\big|\mu^{\rho}_\Omega\big(b^k_i\big)(x)\big|&\le C\cdot\big\|b^k_i\big\|_{L^\infty}
\bigg[\frac{(r^k_i)^{n+1}}{|x-x^k_i|^{n-\rho+1}}+\frac{(r^k_i)^{n+\alpha}}{|x-x^k_i|^{n-\rho+\alpha}}\bigg]
\bigg(\int_{\frac{|x-x^k_i|}{2}}^\infty\frac{dt}{t^{2\rho+1}}\bigg)^{1/2}\\
&\le C\cdot \big\|b^k_i\big\|_{L^\infty}
\bigg[\frac{(r^k_i)^{n+1}}{|x-x^k_i|^{n+1}}+\frac{(r^k_i)^{n+\alpha}}{|x-x^k_i|^{n+\alpha}}\bigg].
\end{split}
\end{equation*}
Note that $\big\|b^k_i\big\|_{L^\infty}\le C2^k$. Hence
\begin{equation*}
\begin{split}
K''_2\le\,& C\sum_{k=k_0+1}^\infty\sum_i2^{kp}\big(r^k_i\big)^{(n+\alpha)p}
\int_{|x-x^k_i|\ge\tau^{{(k-k_0)}/{(n+\alpha)}}\sqrt n r^k_i}\frac{w(x)}{|x-x^k_i|^{(n+\alpha)p}}\,dx\\
&+C\sum_{k=k_0+1}^\infty\sum_i2^{kp}\big(r^k_i\big)^{(n+1)p}
\int_{|x-x^k_i|\ge\tau^{{(k-k_0)}/{(n+\alpha)}}\sqrt n r^k_i}\frac{w(x)}{|x-x^k_i|^{(n+1)p}}\,dx\\
=\,&\mbox{\upshape III+IV}.
\end{split}
\end{equation*}
Let us consider the term III. Recall that $q=p(1+\frac{\alpha}{n})$, then we can deduce
\begin{equation*}
\begin{split}
\mbox{\upshape III}&= C\sum_{k=k_0+1}^\infty\sum_i2^{kp}\big(r^k_i\big)^{nq}
\int_{|y|\ge\tau^{{(k-k_0)}/{(n+\alpha)}}\sqrt n r^k_i}\frac{w^k_i(y)}{|y|^{nq}}\,dy\\
&\le C\sum_{k=k_0+1}^\infty\sum_i
2^{kp}\Big(\tau^{{(k-k_0)}/{(n+\alpha)}}\Big)^{-nq}w^k_i\Big(Q\big(0,\tau^{{(k-k_0)}/{(n+\alpha)}}\cdot r^k_i\big)\Big)\\
&= C\sum_{k=k_0+1}^\infty\sum_i
2^{kp}\Big(\tau^{{(k-k_0)}/{(n+\alpha)}}\Big)^{-nq}w\Big(Q\big(x^k_i,\tau^{{(k-k_0)}/{(n+\alpha)}}\cdot r^k_i\big)\Big),
\end{split}
\end{equation*}
where $w^k_i(x)=w(x+x^k_i)$ is the translation of $w(x)$. It is obvious that $w^k_i\in A_{q}$ whenever $w\in A_{q}$, and $q_{w^k_i}=q_w$. In addition, for $w\in A_q$ with $q>1$, then we can take a sufficiently small number $\varepsilon>0$ such that $q-\varepsilon\ge1$ and $w\in A_{q-\varepsilon}$. Therefore, by using Lemma 2.1 again, we eventually obtain
\begin{equation*}
\begin{split}
\mbox{\upshape III}&\le C\sum_{k=k_0+1}^\infty\sum_i2^{kp}\Big(\tau^{{(k-k_0)}/{(n+\alpha)}}\Big)^{-n\varepsilon}w\big(Q^k_i\big)\\
&\le C\big\|f\big\|^{p}_{WH^p_w}\sum_{k=k_0+1}^\infty\Big(\tau^{{(k-k_0)}/{(n+\alpha)}}\Big)^{-n\varepsilon}\\
&\le C\big\|f\big\|^{p}_{WH^p_w}.
\end{split}
\end{equation*}
For the last term IV, since $w\in A_{p(1+\frac{\alpha}{n})}$ and $0<\alpha\le1$, then we have $w\in A_{p(1+\frac{1}{n})}$. Thus, by using the same arguments as above, we can also prove
\begin{equation*}
\mbox{\upshape IV}\le C\big\|f\big\|^{p}_{WH^p_w}.
\end{equation*}
Combining the above estimates for $K_1$ and $K_2$, and then taking the supremum over all $\lambda>0$, we conclude the proof of Theorem 1.3.
\end{proof}

We finally remark that for any function $f$, a straightforward computation shows that the grand maximal function of $f$ is pointwise dominated by $M(f)$, where $M$ denotes the standard Hardy-Littlewood maximal operator. Hence, by the weighted weak (1,1) estimate of $M$, it is easy to see that the space $L^1_w$ is continuously embedded as a subspace of $WH^1_w$ whenever $w\in A_1$, and we have $\|f\|_{WH^1_w}\le C\|f\|_{L^1_w}$ provided that $w\in A_1$. As a direct consequence of Theorem 1.3, we immediately obtain the following result.

\newtheorem{corollary}[theorem]{Corollary}
\begin{corollary}
Let $0<\rho<n$, $0<\alpha\le1$ and $\Omega\in Lip_\alpha(S^{n-1})$. If $p=1$ and $w\in A_1$, then there exists a
constant $C>0$ independent of $f$ such that
\begin{equation*}
\big\|\mu^{\rho}_{\Omega}(f)\big\|_{WL^1_w}\le C\|f\|_{L^1_w}.
\end{equation*}
\end{corollary}

\end{document}